\documentclass[12pt]{article}

\usepackage{amsfonts}
\usepackage{graphics}
\usepackage{amssymb}

\newtheorem{theorem}{Theorem}[section]

\newtheorem{definition}{Definition}[section]
\newtheorem{proposition}[theorem]{Proposition}

\newtheorem{example}{Example}[section]
\newtheorem{remark}{Remark}[section]

\newcommand{\eop }{ \hfill $\Box$ }

\begin{document}

\begin{center}
{\Large Renormalized-Generalized Solutions for the KPZ Equation \\}

\end{center}

\vspace{0.3cm}

\begin{center}
{\large Pedro Catuogno  and
Christian Olivera} \\

\textit{Departamento de
 Matem\'{a}tica, Universidade Estadual de Campinas, \\ F. 54(19) 3521-5921 � Fax 54(19)
3521-6094\\ 13.081-970 -
 Campinas - SP, Brazil. e-mail: pedrojc@ime.unicamp.br ; colivera@ime.unicamp.br}
\end{center}
\vspace{0.3cm}

\begin{center}
\begin{abstract}
\noindent This work introduces a new notion of solution for the
KPZ equation, in particular, our approach encompasses the Cole-Hopf solution. We 
set in the context of  the distribution theory the  proposed results by Bertini and Giacomin from the mid 90's. 

 \noindent This new approach  provides
 a pathwise notion of solution as well as  a  structured approximation theory. The developments are based on
 regularization arguments from the theory of distributions.
 \end{abstract}
\end{center}


\vspace{0.3cm} \noindent {\bf MSC2000 subject classification:}  46F99, 46F30, 60H15.

\section {Introduction}

\noindent Models for  interface growth have attracted much attention during the last two decades, in particular, these models
have  widespread application in many systems  ( see \cite{Bar} and  \cite{Mea}), such as film growth by vapor or
chemical deposition, bacterial growth, evolution of forest fire fronts, etc. For such systems,
a major effort has been the identification of the scaling regimes and their
classification into universality classes.  Phenomenological equations, selected according to symmetry principles and
conservation laws, are often able to reproduce various experimental data. Among these
phenomenological equations, the one introduced by M. Kardar, G. Parisi and Y. Zhang (\cite{KPZ}) has been
successful in describing properties of rough interfaces. This equation, the KPZ equation,  is also related to Burgers'
equation of turbulence and also directed polymers in random media (see   \cite{BG}).  The KPZ equation
describes the evolution of the profile of the interface, $h(x, t)$,  at position $x$ and time $t$:
\begin{equation}\label{KPZ}
 \left \{
\begin{array}{lll}
    \partial_th(t, x) =   \ \triangle h(t, x) +  \ (\partial_x h(t,x))^{2} + W(t,x)\\
 h(0, x) = f(x)
\end{array}
\right.
\end{equation}
where $W(t,x)$ is a space-time white noise. We refer to \cite{Hairer} for a more detailed historical account of the KPZ equation.

\noindent From a mathematical point of view,  equation (\ref{KPZ}) is ill-posed since the solutions are expected to look
locally like a Brownian motion. The term $  (\nabla h(t,x))^{2}$ cannot make sense in a classical way. The theory of distributions developed  by L. Schwartz  in the 1950s has certain insufficiencies, in particular, the absence of a
well defined multiplication of distributions (see for instance \cite{ober2} and  \cite{Schw2}).

\noindent L. Bertini and G. Giacomin in \cite{BG},  proposed that the correct solutions for the KPZ equation are obtained by taking the logarithm of a
solution of the stochastic heat equation with multiplicative noise. This is known as the Cole-Hopf solution for the KPZ equation.

\noindent The stochastic heat equation with multiplicative noise is the following It\^o equation,
\begin{equation}\label{heat}
 \left \{
\begin{array}{lll}
 dZ   =  \triangle Z \ dt  + Z dW \\
 Z(0, x) = e^{f(x)}.
\end{array}
\right .
 \end{equation}
It is well-posed and its solution is always positive. The evidence for the Cole-Hopf solution is overwhelming even though it satisfies, at a
purely formal level, the equation
\begin{equation}\label{KPZI}
 \left \{
\begin{array}{lll}
    \partial_t h(t, x) =   \ \triangle h(t, x) +  \ (\partial_x h(t,x))^{2} + W(t,x)  - \infty\\
 h(0, x) = f(x)
\end{array}
\right .
\end{equation}
where $\infty$  denotes an infinite constant.

\noindent The above remark indicates that a key problem for the theory of stochastic partial differential
equations is to find an appropriate definition of solution for the KPZ equation
 which, at the same time, incorporates the Cole-Hopf solution.

\noindent The present article constructs spaces of generalized stochastic processes that can be used to
renormalize  the divergent term appearing in (\ref{KPZ}). Moreover, we prove that
the Cole-Hopf solution becomes then a well defined solution to the KPZ equation.  At the level
 of  sequences, our  space of generalized stochastic processes  looks like  a Colombeau space  (see J. F. Colombeau \cite{colb})
and at the level of Schwartz distributions it looks like a generalized stochastic process space in the sense  of
It\^o-Gelfand-Vilenkin (see  \cite{Gel4}
  and \cite{ITO} ). A key ingredient in our study
is the  use of regularization techniques to define nonlinear operations
in the theory of distributions, we refer the reader to  \cite{Miku}, \cite{ober2} and  \cite{Trio} for the background material.

\noindent The article is organized as follows: Section 2 reviews some basic facts on the standard cylindrical Wiener process. Section 3 introduces  the spaces
of renormalized  semimartingales and renormalized distributions. Moreover, we give a notion of random generalized functions.
Section 4 introduces the new concept of  solution for the KPZ equation and proves existence of the solution which coincides with the Cole-Hopf solution.

\section{Cylindrical Brownian motion }

\noindent Let   $\{W(t,\cdot): t\in [0,T]\}$ be a
standard cylindrical Wiener process in $L^{2}(\mathbb{R})$; it is canonically realized as a family of
continuous processes satisfying:
\begin{enumerate}
\item For any $\varphi \in L^2( \mathbb{R})$, $\{W_t (\varphi), t \in [0,T] \}$ is a Brownian motion with variance $t\int\varphi^2(x)~dx$,
\item For any $\varphi_1,\varphi_2 \in L^2( \mathbb{R})$ and  $s, t \in [0,T]$,
\[
\mathbb{E}(W_{s}(\varphi_1)(W_{t}(\varphi_2))=(s \wedge t) \int \varphi_1(x) \varphi_2(x) ~dx.
\]
\end{enumerate}

\noindent Let $\{ \mathcal{F}_t : t \in [0,T]\}$ be the $\sigma$-field generated by the $P$-null sets and the random variables
$W_{s}(\varphi)$, where $\varphi\in \mathcal{D}( \mathbb{R})$ and  $s \in [0,t]$. The predictable $\sigma$-field is the $\sigma$-field in
$[0, T] \times \Omega$  generated by the sets $(s, t]\times A$ where $A\in\mathcal{F}_s$ and $0 \leq s < t \leq T$.

\noindent Let $\{v_j : j\in \mathbb{N} \}$ be a complete orthonormal basis of  $L^{2}(\mathbb{R})$. For any predictable
process $g \in L^{2}(\Omega \times [0, T ], L^2( \mathbb{R}) )$ it turns out that the following series is convergent in
$L^{2}(\Omega, \mathcal{F}, P)$ and the sum does not depend on the chosen orthonormal system:

\begin{equation}\label{inteCy}
\int_{0}^{T}  g_{t} W_{t}  := \sum_{j=1}^{\infty} \int_{0}^{T}  (g_t, v_{j}) dW_t(v_{j}).
\end{equation}

\noindent  We notice that each summand in the above series is a classical It\^o integral with respect to a standard
Brownian motion, and the resulting stochastic integral is a real-valued random variable. The independence of the terms in the series
(\ref{inteCy}) leads to the isometry property

\[
\mathbb{E}(| \int_{0}^{T}  g_{s} ~ dW_{s} |^{2})
   = \mathbb{E}( \int_{0}^{T}  \int | g_s(x) |^{2}~dx~ds).
\]

\noindent See   \cite{Da} and  \cite{Wals} for properties of the cylindrical Wiener motion and
stochastic integration.

\noindent In order to characterize the solution  of the KPZ equation through a procedure  of regularization, we introduce a
mollified version of the cylindrical Wiener process.  Let $\rho :  \mathbb{R} \rightarrow \mathbb{R}$ be an infinitely differentiable function with
compact support such that
$\int \rho(x) \ dx=1$. For all $n\in \mathbb{N}$ we consider the mollifier $\delta_{n}(x-y)=n \rho(n(x-y))$. The mollifier cylindrical Wiener process is defined by:
\begin{equation}
W_{t}^{n}(x):= W_{t}(\delta_{n}(x-\cdot)).
\end{equation}

\noindent It is a straightforward calculation to check that:

\begin{equation}\label{reguCy}
\mathbb{E}[W_{t}^{n}(x) W_{s}^{n}(y)]= s\wedge t  C_{n}(x-y)
\end{equation}
where $C_{n}(x)= \int \delta_{n}(x-u) \delta_{n}(-u) du$.

\noindent The quadratic variation of $W_{t}^{n}(x)$ is given by
\begin{equation}\label{reguCy1}
< W^{n}(x)>_t= Cnt
\end{equation}
where $C=\int \rho^2(-x)dx$.

\noindent We observe that for any $\varphi \in L^2(\mathbb{R})$,

\[
 \lim_{n \rightarrow \infty} \int W^n_t(x) \varphi(x)dx = W_t(\varphi).
\]
In the case that $\varphi$ has compact support the above convergence is a.e..

%\noindent Finally,  throughout this paper we note $\nabla=\frac{d}{dx}$.

\section{Renormalized  Spaces}

\subsection{Renormalized  distributions }
\noindent We denote by $\mathcal{D}([0,T) \times \mathbb{R})$ the space of
the infinitely differentiable functions with
compact support in   $[0,T) \times \mathbb{R}$,  and
$\mathcal{D}^{\prime}([0,T) \times  \mathbb{R})$ its dual. In order to study the KPZ equation at level of distributions, we  introduce a
space of renormalized  distributions.

\begin{definition}\label{defeal} Let
\[
\mathcal{D}_{0}([0,T) \times \mathbb{R}):=\{ \varphi\in \mathcal{D}([0,T) \times \mathbb{R}) :  \varphi=\partial_x \psi \ with \ \psi \in
\mathcal{D}([0,T) \times \mathbb{R})  \}.
\]
\end{definition}

\begin{definition}
A continuous linear functional $\phi: \mathcal{D}_{0}([0,T) \times \mathbb{R}) \rightarrow \mathbb{R}$ is called a renormalized  distribution.
\end{definition}

\begin{remark} We observe that $\mathcal{D}_{0}([0,T) \times \mathbb{R})$ is equal to the set
\[
\{ \varphi\in \mathcal{D}([0,T) \times \mathbb{R}) : \int \varphi(t,x) \ dx=0   \}.
\]
\end{remark}

\noindent The relationship between the theory of distributions and
 probability theory comes from the seminal works of K. It\^o, I. Gelfand and N. Vilenkin,  see \cite{Gel4} and
\cite{ITO}. They gave many significant contributions to the  theory of stochastic process with values in a space of distributions.
This theory has been successful in  the resolution of stochastic partial differential equations driven by space-time white noise,
see for instance  \cite{Wals}. With the purpose to  give sense to the  KPZ equation, we consider random variables with values
in $\mathcal{D}_{0}^{\prime}([0,T) \times \mathbb{R})$. More precisely.

\begin{definition} Let $\mathbf{D}$ be the  space of functions $T :  \Omega \ \rightarrow \  \mathcal{D}_{0}^{\prime}([0,T) \times \mathbb{R})$ such that
$<T,\varphi> $ is  a  random   variable for all  $\varphi \in   \mathcal{D}_{0}([0,T) \times \mathbb{R})$.    \
The elements of $\mathbf{D}$ are called random generalized
functions.
\end{definition}

In order to define the initial condition for the KPZ equation, we need the notion of section of a distribution of S. Lojasiewicz, see
\cite{lo} and \cite{ober2}. We first introduce the space of test functions
\[
\mathcal{D}_{0}( \mathbb{R}):=\{ \varphi\in \mathcal{D}( \mathbb{R}) :  \varphi=\partial_x \psi \ with \ \psi \in
\mathcal{D}( \mathbb{R})  \}
\]
and its dual $\mathcal{D}_{0}^{\prime}(  \mathbb{R})$.

For convenience,  we recall the notion of strict delta net, see \cite{ober2} pp. 55.

\begin{definition}  A strict delta net is a net $\{  \rho_{\varepsilon} : \varepsilon >0 \}$ of $\mathcal{D}([0,T))$ such that it satisfies:
 \begin{enumerate}
\item
$\lim_{\varepsilon \rightarrow 0} supp(\rho_{\epsilon})= \{ 0 \}$.

\item For all $\varepsilon >0$,
 $\int \rho_{\varepsilon}(t)  \ dt =1$.
\item $\sup_{\varepsilon >0} \int |\rho^{\varepsilon}(t)| \ dt < \infty$.
\end{enumerate}
\end{definition}

\begin{definition} A distribution $H   \in  \mathcal{D}_{0}^{\prime}([0,T) \times \mathbb{R})$ has a section
$U \in \mathcal{D}_{0}^{\prime}(  \mathbb{R}) $ at $t=0$ if for all $\varphi \in \mathcal{D}_{0}( \mathbb{R})$ and
all strict delta net $\{  \rho_{\varepsilon} : \varepsilon >0 \}$,

\[
\lim_{\varepsilon \rightarrow 0}<H,  \rho_{\epsilon} \varphi > \ = \ <U,  \varphi >.
\]
\end{definition}

\subsection{Renormalized  Semimartingales }

\noindent We say that a random field $\{ S(t,x) : t \in [0,T],~x \in \mathbb{R} \}$ is a spatially dependent semimartingale if for
each $x\in \mathbb{R}$,
$\{S(t,x): t \in [0,T]\}$ is a semimartingale in relation to the same filtration $\{ \mathcal{F}_t : t \in [0,T] \}$. If
$S(t,x)$ is a $C^{\infty}$-function of $x$ and continuous in $t$ almost surely,  it is called  a $C^{\infty}$-semimartingale.
The vector space of $C^{\infty}$-semimartingales is denoted by $\mathbf{S}$. See \cite{Ku} for a rigorous study of spatially depend semimartingales
and applications to  stochastic differential equations.

\noindent In order to introduce the space of renormalized  semimartingales we consider
$\mathbf{S}^{\mathbb{N}_0}$ the space of sequences of $C^{\infty}$-semimartingales. It is clear that
$\mathbf{S}^{\mathbb{N}_0}$  has the structure of an associative,
commutative differential algebra with the natural operations:
\begin{eqnarray*}
(F_{n})+(G_{n})& = & (F_{n}+ G_{n}) \\
a(F_{n})& = & (aG_{n}) \\
(F_{n})\cdot (G_{n})& = & (F_{n} G_{n}) \\
\partial_x (F_{n})& = & (\partial_x F_{n})
\end{eqnarray*}
where $(F_{n})$ and $(G_{n})$ are in $\mathbf{S}^{\mathbb{N}_0}$ and $a \in
\mathbb{R}$.

\noindent We observe that
\[
\mathbf{N}=\{(G_{n})\in
 \mathbf{S}^{\mathbb{N}_0} \ : \ \mbox{for each $n \in
\mathbb{N}_0$,} \ G_{n} \ \mbox{   does not depend on the
variable } x~~  almost \ surely \}
\]
is a subalgebra
of $\mathbf{S}^{\mathbb{N}_0}$.

\begin{definition}
The space of  renormalized semimartingales  is defined as
\[
\mathbf{G}= \mathbf{S}^{\mathbb{N}_0} / \mathbf{N} .
\]
\end{definition}
Let $(F_n) \in \mathbf{G}$. We will use $[F_n]$ to
denote the equivalence class $(F_n) + \mathbf{N}$.

\begin{example}  The  mollified cylindrical Wiener process $\mathcal{W}:=[W^n]$ belongs to $\mathbf{G}$.
\end{example}

\begin{proposition}\label{deriva} Let $[F_n] \in \mathbf{G}$. Then

\[
\partial_x [F_n]=[\partial_x F_n].
\]

\end{proposition}

\begin{proof} Let $(C_n )\in \mathbf{N}$. Then

\[
\partial_x (F_n+C_n)=(\partial_x F_n+\partial_x C_n)= (\partial_x F_n).
\] \eop
\end{proof}

Now, we introduce the relation of association for renormalized semimartingales.

\begin{definition}\label{ass1} Let $[F_n]$ and $[G_n]$ be renormalized semimartingales. We say
that $[F_n]$ and $[G_n]$ are associated, (denoted by $ [F_n]
\approx [G_n]$) if for all $\varphi \in \mathcal{D}_{0}([0,T) \times \mathbb{R}) $
\[
\lim_{n\rightarrow\infty}<F_n-G_n,  \varphi>=0.
\]
\end{definition}
\noindent We observe that the relation $\approx$ is well defined.
In fact, if $(C_n)\in \mathbf{N}$ and $\varphi \in \mathcal{D}_{0}([0,T) \times \mathbb{R})$ we have that
$\lim_{n\rightarrow\infty}<C_n,  \varphi>=0$. It follows immediately
that $\approx$ is an equivalence relation.

\begin{proposition}\label{equias}
 Let $[F_n]$ and $[G_n]$ be renormalized semimartingales such that $ [F_n]
\approx [G_n]$. Then $ \partial_x^{\alpha}[F_n] \approx \partial_x^{\alpha}[G_n]$
for all $\alpha \in \mathbb{N}$.

\end{proposition}

\begin{proof}
\noindent By integration by parts and hypothesis,
\begin{eqnarray*}
\lim_{n\rightarrow\infty}<\partial_x^{\alpha}F_n -\partial_x^{\alpha}G_n, \partial_x \varphi> &
= & \lim_{n\rightarrow\infty}<F_n -G_n,\partial_x ( (-1)^{\alpha}  \partial_x^{\alpha}\varphi)> \\ & = & 0
\end{eqnarray*}
for all $\varphi\in  \mathcal{D}([0,T) \times \mathbb{R})$. Thus $ \partial_x^{\alpha}
[F_n]\approx \partial_x^{\alpha}[G_n]$. \eop

\end{proof}

\begin{remark} We would like to  remark that our definition of  the space of
renormalized semimartingales is inspired  by Colombeau's theory of nonlinear  generalized functions, see for instance  \cite{Biag}
and \cite{colb}. This theory has been successful in applications to stochastic analysis. We refer the reader to
  \cite{albe}, \cite{albe3},  \cite{CO1}, \cite{CO2}, \cite{CO3}, \cite{ober4} and \cite{Russo} for interesting works on this subject. \end{remark}

\noindent In this context we have the following relation of association between  renormalized semimartingales
and random generalized functions.

\begin{definition}\label{ass1} Let $[F_n]$ be renormalized semimartingales and $T$ be random generalized functions. We say
that $[F_n]$ is  associated with $T$, (denoted by $[F_n]
\simeq T$) if for all $\varphi \in \mathcal{D}_{0}([0,T) \times \mathbb{R})$,
\[
\lim_{n\rightarrow\infty}<F_n,  \varphi>=<T, \ \varphi>.
\]
\end{definition}
\noindent We observe that the relation $ \simeq$ is well defined.
In fact, if $(C_n)\in \mathbf{N}$ and $\varphi \in \mathcal{D}_{0}([0,T) \times \mathbb{R})$ we have that
$\lim_{n\rightarrow\infty}<C_n ,  \varphi>=0$.

\section{Solving the KPZ Equation}

\subsection{Solution  of the  KPZ equation in $\mathbf{D}$}
\noindent The distribution theory of L. Schwartz is the linear calculus of today. It is a cornerstone of modern mathematical analysis with many
applications in economics, engineering and natural sciences. However, it is not possible to define a product for arbitrary distributions with good
properties, see \cite{colb}, \cite{ober2} and  \cite{Schw2}.

%Product of distributions naturally appears in applications, see \cite{colb}, \cite{ober2} and references.

%The distribution theory is essencially a linear theory, specially adapted to solve linear partial differential equations,

%Soon after the introduction of his own theory, L. Schwartz published a
%paper in which he showed an impossibility result, \cite{Schw2} about the
%product of two arbitrary distributions. However, in some applications there
%is need for such a product. Various mathematicians looked for a way around
%the Schwartz impossibility result,  they tried
%to  found  methods to define  the product of two arbitrary distributions. Some
%of them partly solved that issue,  see \cite{Miku}, \cite{calo} , \cite{colb2},  \cite{ober2} and \cite{Trio}.

\noindent A possible approach to define a product of a pair of
distributions is via regularization. That is, approximate one (or both) of the distributions to be multiplied by smooth functions,
multiply this approximation by the other distribution, and pass to
the limit (see for instance \cite{Miku},  \cite{calo},  \cite{ober2} and \cite{Trio}).  Inspired by  ideas of regularization and passage to the corresponding limit, we
introduce a new concept of solution for the KPZ equation.

%We rely on this spirit of regularization
%to give sense to the KPZ equation.

\begin{definition}\label{solu3}
We say that $H\in  \mathbf{D}$
is a {\it renormalized-generalized  solution} of the equation  (\ref{KPZ})
if, almost surely, 
\begin{enumerate}
\item There exists a sequence of $C^{\infty}$-semimartingales $\{ H_n : n \in \mathbb{N} \}$ such that
$ H=\lim_{n \rightarrow \infty}H_n$
and there exists
$\lim_{n\rightarrow\infty}   (\partial_x H_n)^{2}$
in $\mathcal{D}^{\prime}_{0}([0,T) \times  \mathbb{R})$.

\item For all $\varphi \in \mathcal{D}_{0}([0,T) \times  \mathbb{R})$,
\[
< H, \partial_t\varphi > +   < \triangle H, \varphi >  +  < (\partial_x H )^{2},  \varphi >  +  \int_{0}^{T}  \varphi(t,\cdot) dW_t=0
\]
where $(\partial_x H )^{2}:=\lim_{n\rightarrow\infty} (\partial_x H_n)^{2}$.

\item The section of $H$ at $t=0$ is equal to $f$.
\end{enumerate}

\end{definition}

\begin{theorem}\label{teoSolu2} Let  $f\in  C_{b}^{\infty}(\mathbb{R})$ and $Z$ be a solution of the stochastic heat equation (\ref{heat}). Then $H=ln~Z$
 is a renormalized-generalized solution of (\ref{KPZ}).
\end{theorem}
\begin{proof}  Let us denote by $H_{n}(t,x)$ the process $ln \  Z_{n}(t,x)$, where $Z_n$ is the solution of the regularized stochastic heat equation in
the It\^o sense
\begin{equation}\label{heat2}
 \left \{
\begin{array}{lll}
dZ_{n}  & = & \triangle Z_{n} \ dt  + Z_{n} \ dW^{n}, \  \\
Z_n(0,x) & = &  e^{f(x)}.
\end{array}
\right .
\end{equation}
See \cite{BC} for an exhaustive study of the solutions of this equation.

% We claim that $H=ln \  Z$ is a renormalized-generalized solution of the KPZ equation.   In fact,

\noindent By It\^o formula and (\ref{reguCy1}),

\begin{equation}\label{KPZA2}
H_{n} =   f  +
 \int_{0}^{t}  \triangle H_{n} \ ds
+   \int_{0}^{t} (\partial_x H_{n})^{2} \ ds
+     W_{t}^{n} - \frac{ C}{2} \ nt .
\end{equation}

\noindent Multiplying (\ref{KPZA2}) by $\partial_t \varphi(t,x)$, where $\varphi \in \mathcal{D}_{0}([0,T) \times  \mathbb{R})$,
and integrating in $[0,T)\times \mathbb{R}$ we obtain that
\begin{equation}\label{KPZA3}
< H_{n}, \partial_t \varphi > +
   < \triangle H_{n}, \varphi > +  < (\partial_x H_{n})^{2},  \varphi >
+    \int_{0}^{T} ( \varphi(t,\cdot) * \delta_{n} )  \  dW_{t} =0  .
\end{equation}

\noindent We observe that   $ Z_{n}$ converge to  $Z$    uniformly on compacts of $(0,T) \times \mathbb{R}$, see  L. Bertini and
N. Cancrini \cite{BC}, Theorem 2.2. Thus

  \begin{equation}\label{C1}
 \lim_{n \rightarrow \infty}< H_{n}, \partial_t \varphi> = < H, \partial_t \varphi>
   \end{equation}
and
  \begin{equation}\label{C2}
\lim_{n \rightarrow \infty}< \triangle H_{n}, \varphi> =   < \triangle H, \varphi >.
\end{equation}

\noindent We recall that $\int_{0}^{T}   \varphi(t,\cdot) \ dW_{t}$ defines a continuous linear functional from $\mathcal{D}([0,T) \times \mathbb{R})$ to
$\mathbb{R}$, see \cite{Sch}. Then

\begin{equation}\label{C3}
\lim_{n \rightarrow \infty}  \int_{0}^{T} (\varphi(t,\cdot) * \delta_{n})  \  dW_{t} = \int_{0}^{T}   \varphi \ dW_{t}.
\end{equation}

\noindent  From the equation (\ref{KPZA3})
and the  convergences   (\ref{C1}), (\ref{C2}) and (\ref{C3}) we have that for all $ \varphi\in \mathcal{D}_{0}([0,T) \times \mathbb{R})$,

\[
 \int_{0}^{T} \int_{ \mathbb{R}}   (\partial_x H_{n})^{2} \  \varphi(t,x)  \ dtdx \
\]
converges and define a linear functional in $\mathcal{D}_{0}^{\prime}([0,T) \times \mathbb{R})$ . Thus the  nonlinearity
\[
<(\partial_x H )^{2} , \varphi> := \lim_{n \rightarrow \infty} \int_{0}^{T} \int_{ \mathbb{R}}   (\partial_x H_{n})^{2} \  \varphi(t,x)  \ dtdx \
\]
is well defined.

Since the initial data is bounded we have the
continuity

\noindent Since the initial data is bounded we have the
continuity  of $ Z(t,x) $ in $[0,T] \times \mathbb{R}$, see Theorem 3.1 of 
\cite{LeDalang} and references  \cite{Prop} and \cite{Shiga}. Thus 

\[
\lim_{\varepsilon\rightarrow 0} \int \int_{0}^{T} ln \  Z(t,x) \  \rho_{\varepsilon}(t) \ dt \   \varphi(x) \ dx   =    \int f(x) \  \varphi(x)\ dx   ,
\]

 \noindent  for all $\varphi \in \mathcal{D}_{0}( \mathbb{R})$ and  for all strict delta nets. Thus  we conclude that $H$ is a renormalized-generalized  solution   for the  problem (\ref{KPZ}).

\end{proof}

\begin{remark}
The solution $Z(t,x)$ of the stochastic heat equation (\ref{heat}) is understood in a mild sense (see for example \cite{BC}, \cite{Da} and \cite{TESS}).
\end{remark}

\subsection{Solution  of the  KPZ equation in $\mathbf{G}$}

\begin{definition}\label{solu1}
We say that $F\in \mathbf{G}$
is a {\it renormalized  solution} of the KPZ-equation  (\ref{KPZ})
if
\[
F =   f
+ \ \int_{0}^{t}  \triangle F \ ds
+  \ \int_{0}^{t} (\partial_x F)^{2} \ ds
+  \  \mathcal{W}
\]
where the equality holds in $ \mathbf{G}$.

\end{definition}

\begin{theorem}\label{teoSolu} For every   $f\in  C_{b}^{\infty}(\mathbb{R})$
there exist a renormalized solution $F$ in $ \mathbf{G}$
 for the KPZ-equation (\ref{KPZ}).
\end{theorem}
\begin{proof}
\noindent  Let $Z_n$ be the solution of the regularized stochastic heat equation (\ref{heat2}). We denote $ln \  Z_{n}(t,x)$ by $F_{n}(t,x)$
and $[F_{n}]$ by $F$. We claim that $F$ is a renormalized solution of the  KPZ equation. In fact,  by the It\^o formula we have that
\begin{equation}\label{KPZA}
F_{n} =   f
+ \ \int_{0}^{t}  \triangle F_{n} \ ds
+  \ \int_{0}^{t} (\partial_x F_{n})^{2} \ ds
+  \   W_{t}^{n} - \frac{C}{2} \ n t.
\end{equation}

\noindent Taking equivalence classes, we obtain

\[
F =   f
+ \ [ \int_{0}^{t}  \triangle F_{n} \ ds ]
+  \ [\int_{0}^{t} (\partial_x F_{n})^{2} \ ds ]
+  \  \mathcal{W}.
\]

\noindent We need to check  that

\begin{enumerate}
\item[a)]   \ $ [\int_{0}^{t} (\partial_x F_{n})^{2} \ ds ]= \int_{0}^{t} (\partial_x [F_{n}])^{2} \ ds $

\item[b)] \ $[ \int_{0}^{t}  \triangle F_{n} \ ds ]=  \int_{0}^{t}  \triangle \ [ F_{n}] \ ds $.

\end{enumerate}

\noindent $\mathrm{a)}$ We observe  that
\[
\int_{0}^{t} (\partial_x F_{n})^{2} \ ds + D_{n} - \int_{0}^{t} (\partial_x ( F_{n}+ C_{n})  )^{2} \ ds
\]
is equal to
\[
\int_{0}^{t} (\partial_x F_{n})^{2} \ ds + D_{n} - \int_{0}^{t} (\partial_x  F_{n}  )^{2} \ ds=D_{n}.
\]

\noindent $\mathrm{b)}$ The  proof is analogous to  $\mathrm{a)}$.
\end{proof}

\begin{remark} We observe from the proof of the Theorem \ref{teoSolu}   that   $F=[\ln Z_{n}]$ is a renormalized
solution of the KPZ equation. Moreover, we have that for all  $\varphi \in \mathcal{D}_{0}([0,T) \times \mathbb{R})$ it holds
\[
\lim_{n\rightarrow\infty}<\ln Z_{n},  \varphi>=<\ln Z, \ \varphi>.
\]
\noindent Hence we get $[\ln Z_n]
\simeq \ln Z$ in the sense of Definition
\ref{ass1}.
\end{remark}

\section{ End Comments and Future Work}

\subsection{ Others Approaches} Very recently,  two  other ways to make sense of the Cole-Hopf solution were given.

\noindent S. Assing, in   \cite{AS2}, introduces a weak solution, in a probabilistic sense,  for the conservative version of the  equation
(\ref{KPZ}). The idea is to  approximate the Cole-Hopf solution by the density fluctuations in weakly asymmetric exclusion.
In  \cite{Jara}, P. Gon\c{c}alves and M. Jara considered a similar type of solution. 

   \noindent M. Hairer,  in  \cite{Hairer}, obtained a complete existence and uniqueness result for KPZ.
In this remarkable paper he uses the theory of controlled rough paths to give meaning
to the nonlinearity and a careful analysis, via systematic use of Feynman diagrams, of
the series expansion of the solution.  He was able to show that the solution he identifies coincides with the Cole-Hopf solution.
 However, the results of \cite{Hairer}  are for KPZ on $[0,2\pi]$ . with periodic boundary conditions, and extending them to $\mathbb{R}$ remains an open problem.

\subsection{ Solution of the KPZ Equation for $d>1$.}

\noindent The stochastic heat equation
in $\mathbb{R}^{d}$,  is given by

\begin{equation}
 \left \{
\begin{array}{lll}
\partial_t Z  & = & \triangle Z   + Z \ W \  \\
Z(0,x) & = &  e^{f(x)}
\end{array}
\right.
\end{equation}
where $W(t,x)$ is the space-time white noise in $\mathbb{R}^d$.

\noindent We observe that  the Cole-Hopf solution
does not make sense because the solution of the  stochastic heat equation
is not a standard  stochastic process. It is
realized as a generalized stochastic process in the space of stochastic Hida distribution (see for example \cite{HOBZ}).

 \noindent Thus, we  expect solutions of  KPZ equation
  only in the sense of  Colombeau's generalized functions. Stochastic processes whose paths are generalized functions are considered in
\cite{albe} and \cite{albe3}, and are used for
solving some classes of nonlinear stochastic equations. For a new approach see
\cite{CO1}, \cite{CO2} and \cite{CO3}.

%\noindent Finally, we mention that another renormalization  method is via neutrix calculus. This calculus was developed by J. Van der Corput
%and J. Hadamard in connection with asymptotic series and divergent integrals. We refer to the reader to \cite{Van}.

\end{document}